\documentclass[12pt]{article}
\usepackage{amsmath, amscd, amssymb, latexsym, epsfig, color, amsthm}
\pdfoutput=1
\newtheorem{theorem}{Theorem}[section]

\newtheorem{corollary}[theorem]{Corollary}

\newtheorem{lemma}[theorem]{Lemma}

\theoremstyle{definition}

\newtheorem{example}[theorem]{Example}
\newtheorem{remark}[theorem]{Remark}

\newenvironment{proofof*}[1]%without \qed symbol
{\begin{trivlist}\item {\it Proof of {#1}.}}{\end{trivlist}}

\newcommand\bcc{simplicial poset}
\newcommand\dns{discrete normal surface}

\newcommand\ZZ{\mathbb{Z}}
\newcommand\CC{\mathbb{C}}

\title{The average dual surface of a cohomology class and minimal simplicial decompositions of infinitely many lens spaces}
\author{Ed Swartz \thanks{Parially supported by NSF grant DMS-1200478}}

\begin{document}
\maketitle

\begin{abstract}
  Discrete normal surfaces are normal surfaces whose intersection with each tetrahedron of a triangulation has at most one component.  They are also natural Poincar\'e duals to $1$-cocycles with $\ZZ/2\ZZ$-coefficients.  For  a fixed cohomology class in a simplicial poset the average Euler characteristic of the associated \dns s only depends on the $f$-vector of the triangulation.  As an application we determine the minimum simplicial poset representations, also known as crystallizations, of lens spaces $L(2k,q),$ where $2k=qr+1.$  Higher dimensional analogs of \dns s are closely connected to the Charney-Davis conjecture for  flag spheres.  
\end{abstract}

Analyzing compact three-manifolds by cutting them into pieces, in particular tetrahedra, has a long and successful history.  Depending on the author, a ``triangulated three-manifold" can have several different meanings.  At one extreme are abstract simplicial complexes where a face is completely determined by its vertices and a given three-manifold $M$ is triangulated by an abstract simplicial complex $\Delta$ if the geometric representations of $\Delta$ are homeomorphic to $M.$  At the other extreme are the face identification schemes, sometimes called singular triangulations,   most commonly used in modern algorithmic low-dimensional topology.  In a singular triangulation the {\it interiors} of the cells are open simplices.  Giving the closed cells more flexibility may allow one to present $M$ in a very succinct manner.  See, for instance \cite{Mat}.  In between these two  are \bcc s.   Here the {\it closed} cells are simplices, but more than one face can have the same set of vertices.  So in this setting two vertices and two edges are sufficient to triangulate a circle.  A basic result is that any $d$-dimensional closed connected PL-manifold  can be given a \bcc \ triangulation with $d+1$ vertices, the minimum possible  \cite{Pez}, \cite{FG}.  

 In all three cases one of the fundamental  problems is to determine the smallest possible triangulations of a given three-manifold.  
 In \cite{JRT} Jaco, Rubenstein and Tillmann produced the first infinite family of irreducible three-manifolds whose minimal presentation using singular triangulations could be proven. Inspired by their ideas we accomplish the same for \bcc s which are  lens spaces of the form $L(2k,q)$ with $2k=qr+1.$    
 Along the way we will find that for a \bcc \ $\Delta$ the Euler characteristic of the average \dns \ dual to a fixed cohomlogy class $\phi$ in $H^1(\Delta;\ZZ/2\ZZ)$ is independent of $M$ and $\phi.$ It only depends on the $f$-vector of  $\Delta!$  
 
 After setting notation in Section \ref{notation},  the precise meaning of the previous sentence is explained in Section \ref{average}.  Then we show how to use this to prove that minimal simplicial poset representations of $L(2k,q)$ with $2k=qr+1$ have $4(q+r)$ tetrahedra.  Along the way we will see a close connection to the Charney-Davis conjecture for flag spheres.

\section{Notation} \label{notation}

Regular CW-complexes in which all closed cells are combinatorially simplices have appeared under a variety of names.  These include semi-simplicial complexes \cite{EZ}, Boolean cell complexes \cite{BWe}, and simplicial posets \cite{St}.  The last is the most frequent in the combinatorics literature, and since we will be concerned with questions of an enumerative nature we will use it from here on.  In any case, the reader will be well-served with the idea that simplicial posets are analogous to abstract simplicial complexes where a set of vertices may determine more than one face.

Throughout $\Delta$ is a \bcc.    As usual $f_0, f_1, f_2$ and $f_3$ will refer to the number of vertices, edges, triangles and tetrahedra of the complex and in general $f_i$ is the number of $i$-dimensional simplices.   We use $L(p,q)$ to stand for the lens space given by $S^3/(\ZZ/p\ZZ)$ where $S^3 = \{(z_1, z_2) \in \CC^2 :|z_1|=|z_2|=1\}$ with $\ZZ/p\ZZ$ as the $p^{th}$-roots of unity acting by $e^{\frac{2\pi i}{p}} \cdot (z_1, z_2) = (e^{\frac{2\pi i}{p}} z_1, e^{\frac{2q\pi i}{p}} z_2).$

A {\it crystallization} of a three-manifold $M$ without boundary is a \bcc \ with exactly four vertices which is homeomorphic as a topological space to $M.$ A well-known crystallization of $L(p,q)$ is formed by first taking the join of two circles each of which consists of $2p$ vertices and edges,  then quotient out by the $\ZZ/p \ZZ$ action.  We will call this the {\it standard} crystallization of $L(p,q).$ For future reference we observe that the standard crystallization of $L(p,q)$ has $4p$ tetrahedra. 

 Except where otherwise noted, all   chains, cochains and their corresponding homology and cohomology groups will be with $\ZZ/2\ZZ$-coefficients.

A surface $S$ contained in $\Delta$ is {\em normal} if for every tetrahedron $T$ of $\Delta$ each component of $S \cap T$ is combinatorially equivalent to one the three possibilities in Figure \ref{normal surface} (taken  from \cite{JRT}).  Suppose that $S$ is a normal surface such that for each $T$ the intersection $S \cap T$ has at most one component. In \cite{Spr} Spreer calls these types of normal surfaces  {\em discrete normal surfaces}. Given a \dns \ $S$ every triangle  of $\Delta$ intersects $S$ in either $0$ or $2$ edges.  Hence, if we define a function $\psi:C_1(\Delta) \to \ZZ/2 \ZZ$ by $\psi(e)$ is one if $e$ intersects $S,$ and zero otherwise, then $\psi$ is a $1$-cocycle of $\Delta.$  Conversely, if $\psi$ is a $1$-cocycle of $\Delta,$ then we can easily produce $S_\psi$ so that $S_\psi$ is a \dns \ and $\psi$ is the $1$-cocycle given by the previous construction applied to $S_\psi. $ Evidently $S_\psi$ as defined above is unique up to combinatorial equivalence.   We note that $S_\psi$ is a $\ZZ/2 \ZZ$-Poincar\'e dual of $\psi.$  We also note that if $\psi$ is the zero $1$-cocycle, then $S_\psi$ is the empty set which we consider to be a \dns \ of Euler characteristic zero.  

\begin{figure} 
\begin{center}
 \scalebox{0.80}[0.80]{\includegraphics{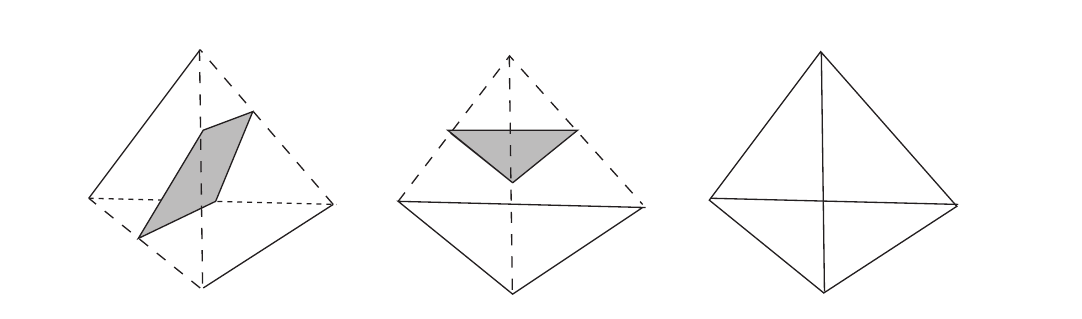}}
 \end{center}
  \caption{Discrete normal surfaces} \label{normal surface}
\end{figure}

Given a discrete normal surface $S=S_\psi$  edges $e$ with $\psi(e)=0$ are called $\psi$-{\it even} edges.  Similarly, edges $e$ such that $\psi(e)=1$ are called $\psi$-{\it odd} edges.  

\section{The average discrete normal surface} \label{average}

For a  one-cocycle $\psi$ let $\Delta_\psi$ be the subcomplex of $\Delta$ obtained by removing  the $\psi$-even edges.  Equivalently, the faces of $\Delta_\psi$ are the faces of $\Delta$ which do not contain any $\psi$-odd edges.  

\begin{lemma}
 Let $\Delta$ be a \bcc \ of dimension $d.$    For any cohomology class $[\phi]$ in $H^1(\Delta)$ the average Euler characteristic of $\Delta_\psi$ for all cocycles $\psi$ in $[\phi]$ is
  $$\displaystyle\sum^d_{j=0}  (-\frac{1}{2})^j f_j = f_0 - (1/2)f_1 + \dots + (-1)^d (1/2^d) f_d.$$
\end{lemma}

\begin{proof}
Let $n$ be the number of vertices of $\Delta$ and $c$ the number of components of $\Delta.$ Choose a representative $1$-cocycle $\sigma \in [\phi].$     Consider
\begin{equation} \label{Z}
Z = \sum_{u \in C^0(\Delta)} \chi(\Delta_{\sigma + \delta^0(u)}).
\end{equation} 
\noindent The coboundary map $\delta^0:C^0 \to C^1$ has a $c$-dimensional kernel, so $Z$ counts each cocycle in $[\phi] \ 2^c$-times.  However,  this has no effect on the fact that the average value of $\chi(\Delta_\psi)$ is $Z/2^n.$  

Fix a face $F$ of $\Delta.$  What is the contribution of $F$ to $Z?$ By definition, it is $ (-1)^{\dim F}M_F,$ where $M_F$ is the number of $u \in C^0$ such that all of the edges of $F$ evaluate to zero in $\sigma+\delta^0(u).$ Since $F$ is a simplex any cocycle is acyclic when restricted to $F.$  Hence there are $u' \in C^0$ so that $\sigma + \delta^0(u')$ is zero on the edges of $F.$  By fixing such a $u'$ and viewing (\ref{Z}) as a sum over $\chi(\Delta_{\sigma+\delta^0(u'+u)}),$ we see that $M_F = 2 \cdot 2^{n-(\dim F +1)}.$  Rewriting (\ref{Z}) as a sum over all faces gives the desired result.

\end{proof}

Formula (\ref{psi to delta})  in the following corollary is  from \cite{JRT}.

\begin{corollary}  \label{average ec}
  Let $\Delta$ be a  \bcc \ whose geometric realization is a closed three-manifold.    If $[\phi] \in H^1(\Delta)$, then the average of $\chi(S_\psi)$ over all cocycles $\psi$ in $[\phi]$ is 
  $$\frac{5 f_0 - f_1}{8}=\frac{4f_0-f_3}{8}.$$
\end{corollary}

\begin{proof}
  Let $\psi$ be a cocycle in $[\phi]$  and let $N_1$ be the complement of a small regular neighborhood $N_0$ of $S_\psi$ in $\Delta.$   Then $N_1$  is homotopy equivalent to $\Delta_\psi$.  Hence,
  \begin{equation} \label{psi to delta}
  2 \chi(S_\psi) = 2 \chi(N_0) =  \chi(\partial N_0) =\chi(\partial N_1) = 2 \chi ( N_1) = 2 \chi (\Delta_\psi).
  \end{equation}
  By the previous lemma the average Euler characteristic of the $\Delta_\psi$, and hence the \dns s $S_\psi,$ is $f_0-(1/2) f_1 + (1/4)f_2 - (1/8) f_3.$    Apply the well-known formulas $f_3 = f_1 - f_0$ and $ f_2 = 2 (f_1 - f_0).$ 
    \end{proof}
    
    \begin{remark}
    An $f$-vector formula for the average Euler characteristic of $S_\psi$ when $\Delta$ is a  three-dimensional normal pseudomanifold can be obtained by using the fact that \\ $\chi(\partial N_1) = 2 \chi(N_1) - 2 \chi(\Delta).$
    \end{remark}
    
    \begin{remark}
    In  higher odd dimensions the reasoning of the proof of Corollary \ref{average ec} carries through without change to the combinatorial slicings of \cite{Spr} (see also \cite{Kuh}) as they are the analogs of \dns s corresponding to cocycles cohomologuous to zero in $H^1(\Delta).$  In particular, the Charney-Davis conjecture \cite{CD} for flag PL-spheres is equivalent to the statement that the average Euler characteristic of a combinatorial slicing is greater (in dimensions congruent to one mod four) or less (in dimensions congruent to three mod four) than two.     
    \end{remark}
    
       As an example of possible applications of Corollary \ref{average ec} we consider a simple example.    
    \begin{example}
    Let $\Delta$ be the boundary of an $11$-vertex two-neighborly $4$-polytope.    In this case $5 f_0 - f_1  = (55 - 55)/8=0.$ In addition, each vertex link occurs twice as a discrete normal surface and has Euler characteristic two.  Hence  $\Delta$ must contain a \dns \ $S$ with negative Euler characteristic.  As nonorientable closed surfaces do not embed in the three-sphere, $S$ must be orientable with genus at least two.  
         \end{example}

Our other application of Corollary \ref{average ec} is to determine the size of a minimal crystallization of $L(2k,q)$ whenever $2k = rq+1.$  In other words, $2k-1=qr,$ with $ q$ and $r$ odd positive integers.  In preparation, we recall Bredon and Wood's main result concerning which nonorientable surfaces embed in $L(2k,q).$  The following theorem is implied by   \cite[Theorem 6.1]{BW}.

\begin{theorem} \cite{BW} \label{embed}
Assume that $2k=qr+1, \ q$ and $r$ odd positive integers.  Then a closed nonorientable  surface $S$ embeds in $L(2k,q)$ if and only if its Euler characteristic is $\frac{4-q-r}{2}-2i,$ with $i$ a nonnegative  integer. 
\end{theorem}

\begin{theorem}  \label{lens complexity}
Suppose $2k=qr+1,$ with $q$ and $r$ positive odd integers.  Any minimal crystallization of $L(2k,q)$ has $4(q+r)$ tetrahedra.
\end{theorem}

\begin{proof}
 The theorem is well-known for $k=1,$ so we assume $k \ge 2.$  Let $\Delta$ be a minimal crystallization of $L(2k,q).$  The recent preprints by Casali and Cristofori \cite{CC} and (independendly) Basak and Datta \cite{BD}  produce crystallizations of $L(2k,q)$ with $4(q+r)$ facets. Hence,  $f_3(\Delta) \le 4(q+r).$  So it remains to prove the reverse inequality.    Let $\phi$ be  the nontrivial element of $H^1(\Delta).$ Applying Corollary \ref{average ec} we can find a cocycle $\psi$ in $[\phi]$ such that 
\begin{equation} \label{poset inequality}
4 f_0 - f_3 \le 8 \chi(S_\psi).
\end{equation}
    Since $H_2(L(2k,q); \ZZ)$ is zero $S_\psi$ is not orientable.  By Theorem \ref{embed} the Euler characteristic of the nonorientable components of $S_\psi $ sum to at most $(4-q-r)/2.$  What about orientable components?  Except for spheres, these components do not increase the Euler characteristic of $S_\psi.$ We claim that $S_\psi$ has no sphere components.  Suppose $S_\psi$ has a sphere component. Then $\Delta - S_\psi$ has at least two components, so $\Delta_\psi$ also has at least two components.  That implies that there exist vertices $v_1, v_2$ so that $\psi(e) = 1$ for all edges $e$ between $v_1$ and $v_2.$   Let $X$ be the subcomplex of $\Delta$ spanned by $v_1$ and $v_2.$  The natural inclusion map from $H_1(X) \to  H_1(\Delta)$ is surjective.  See, for instance, \cite{Klee}.  Of course, this is impossible since $\psi$ is nontrivial in cohomology but evaluates to zero on all generators of $H_1(X).$

Since $\chi(S_\psi) \le (4-q-r)/2,$ equation (\ref{poset inequality}) implies that $16-f_3 \le 4(4-q-r)$ so $f_3 \ge 4(q+r)$ as required.

 \end{proof}
\noindent  When $q=1$ and $r=2k-1$ the standard crystallization of $L(2k,1)$ has the desired number of tetrahedra.  In other cases more sophistication is required.  See  \cite{BD} and \cite{CC}.

\begin{remark}
Using 1-dipole moves and connected sum with a balanced sphere whose $h$-vector is $(1,0,2,0,1)$ it is possible to show that determining the minimal crystallization of a closed $3$-manifold $M$ is equivalent to determining all possible $f$-vectors of balanced simplicial posets homeomorphic to $M.$  See \cite{FG} for an explanation of dipole moves, \cite{St} for a balanced sphere with $h$-vector $(1,0,2,0,1)$,  and \cite{Klee} for a recent definition of balanced simplicial poset.  
\end{remark}

\noindent {\it Acknowledgements} \\
Previous versions of the proof of Corollary \ref{average ec} made it clear that this work was originally guided by \cite{JRT}. The author became aware of \cite{JRT} as a result of Stephan Tillmann's lecture at Oberwolfach's Triangulations workshop in 2012.   Later correspondence with Jonathan Browder and Jonathan Spreer eventually inspired the current presentation.

 \small{

}

\end{document}